\documentclass{amsart}

\usepackage{color,graphicx,amssymb,latexsym,amsfonts,txfonts,amsmath,amsthm}
\usepackage{pdfsync,enumitem}
\usepackage{amsmath,amscd}
\usepackage[all,cmtip]{xy}
\usepackage{caption,subcaption}

\usepackage{hyperref}
\hypersetup{
    colorlinks=true,       
    linkcolor=blue,          
    citecolor=blue,        
    filecolor=blue,      
    urlcolor=blue           
}

\newtheorem{theo}{Theorem}                     
\newtheorem{prop}{Proposition}                  
\newtheorem{coro}{Corollary} 
\newtheorem{lemm}{Lemma}
\theoremstyle{remark}                  
\newtheorem{rema}{\bf Remark}

\begin{document}

\title{The structure of MDC-Schottky extension groups}

\author{Rub\'en A. Hidalgo}
\address{Departamento de Matem\'atica y Estad\'{\i}stica, Universidad de La Frontera. Temuco, Chile}
\email{ruben.hidalgo@ufrontera.cl}
\thanks{Partially supported by projects Fondecyt 1230001 and 1220261}

\begin{abstract} 
Let $M^{0}$ be a complete hyperbolic $3$-manifold whose conformal boundary is a closed Riemann surface $S$ of genus $g \geq 2$.  If $M=M^{0} \cup S$, 
then let ${\rm Aut}(S;M)$ be the group of conformal automorphisms of $S$ which extend to hyperbolic isometries of $M^{0}$. If the natural homomorphism at fundamental groups, induced by 
the natural inclusion of $S$ into $M$, is not injective, then it is known that $|{\rm Aut}(S;M)| \leq 12(g-1)$.
If $M$ is a handlebody, then it is also known that the upper bound is attained. 

In this paper, we consider the case when $M$ is homeomorphic to the connected sum of $g \geq 2$ copies of $D^{*} \times S^{1}$, where $D^{*}$ denotes the punctured closed unit disc and $S^{1}$ the unit circle. In this case, we obtain that: (i) if $g=2$, then $|{\rm Aut}(S;M)| \leq 12$ and the equality is attained, this happening for ${\rm Aut}(S;M)$ isomorphic to the dihedral group of order $12$, and (ii) if $g \geq 3$, then $|{\rm Aut}(S;M)|<12(g-1)$, in particular, the above upper bound is not attained.
\end{abstract}

\maketitle

\section{Introduction}
Let $M^{0}$ be a complete hyperbolic $3$-manifold, whose conformal boundary is a closed Riemann surface $S$ of genus $g \geq 2$. 
In this case, $M=M^{0} \cup S$ is a bordered $3$-manifold that does not need to be compact. In \cite{Schwarz}, Schwarz observed that the group ${\rm Aut}(S)$, of conformal automorphisms of $S$, is  finite and, in \cite{Hurwitz}, Hurwitz obtained that ${\rm Aut}(S) \leq 84(g-1)$.  Let us denote by ${\rm Aut}(S;M)$ the subgroup of ${\rm Aut}(S)$ consisting of those conformal automorphisms that extend as hyperbolic isometries of $M^{0}$. 

The embedding $\iota:S \hookrightarrow M$ induces a homomorphism $\iota_{*}:\pi_{1}(S,*) \to \pi_{1}(M,*)$. If the kernel $\ker(\iota_{*})$ is trivial, then the order of ${\rm Aut}(S;M)$ can be as big as possible. For instance, in \cite{GZ}, Gradolato and Zimmermann provided examples, with $M$ compact, such that 
$|{\rm Aut}(S;M)| = 84(g-1)$ (in these examples, ${\rm Aut}(S;M)={\rm Aut}(S) \cong {\rm PSL}_{2}(q)$, where $q \leq 1000$ and $q \notin \{7,27\}$).
If $\ker(\iota_{*})$ is not trivial, then the elements of ${\rm Aut}(S;M)$ will keep invariant a non-empty collection of pairwise disjoint simple closed geodesics on $S$. In that case, the quotient orbifold $S/{\rm Aut}(S;M)$ cannot be of triangular signature, so (by the Riemann-Hurwitz formula) $|{\rm Aut}(S;M)| \leq 12(g-1)$. We may wonder if, in this situation, the upper bound $12(g-1)$ is attained for infinitely many values of $g$.

Let $\Gamma$ be a Kleinian group such that ${\mathbb H}^{3}/\Gamma=M^{0}$ and $\Omega/\Gamma=S$, where $\Omega \subset \widehat{\mathbb C}$ is the (non-empty) region of discontinuity of $\Gamma$. In terms of $\Gamma$, we may identify every subgroup of ${\rm Aut}(M;S)$ with the quaotient group $K/\Gamma$, where $K$ is a suitable Keinian group containing $\Gamma$ as a normal subgroup of finite index. 
So, we are interested in: (i) to find upper bounds for the index of $\Gamma$ in those possible groups $K$, (ii) to see if there is one of such Kleinian groups $K$ containing $\Gamma$ as a maximal finite index normal subgroup, and (iii) to provide a description (in terms of the Klein-Maskit combination theorems \cite{Maskit:book, Maskit:CombIV}) of these Kleinian groups $K$.

If, for example, $M$ is a handlebody of genus $g \geq 2$, then $\Gamma$ is a Schottky group of rank $g$. In \cite{Zimmermann3}, Zimmermann proved that the upper bound  $12(g-1)$ is attained
for infinitely many values of $g$  (see \cite{Hid:Max} for a Schottky group description of these maximal order cases). A description, in terms of the Klein-Maskit combination theorems, of those Kleinian groups that are finite index normal extensions of Schottky groups, was provided in \cite{Hidalgo:virtualSchottky}.

At this point, we may wonder if the upper bound $12(g-1)$ can be realized by infinitely many values of $g \geq 2$ for the other cases. 
In this paper, we proceed to provide examples of $M$ for which the upper bound $12(g-1)$ is not attained for $g \geq 3$. Before to state our main theorem, we need some definitions.

We will say that $M$ is an MCD-handlebody of genus $g$ if it is homeomorphic to the connected sum of $g$ copies of $D^{*} \times S^{1}$, where $D^{*}$ denotes the punctured closed unit disc and $S^{1}$ is the unit circle. In this case, the Kleinian group $\Gamma$ is a free product, in the sense of the Klein-Maskit combination theorems \cite{Maskit:book, Maskit:CombIV}, of $g$ double parabolic groups. We name each of these type of groups as an MDC-Schottky group (some authors also called them Schottky-type groups of type $(0,0,g)$).

A Kleinian group that contains an MDC-Schottky as a finite index normal subgroup is called an MDC-Schottky extension group. In Theorem \ref{estructura}, we provide a description, in terms of the Klein-Maskit combination theorems, of MDC-Schottky extension groups. As a consequence of such a description, in Theorem \ref{cor1}, we provide an explicit description of the group ${\rm Aut}(S;M)$, for the case of a MCD-handlebody.  In particular,  we obtain the following (see Corollary \ref{maxteo1}).

\begin{theo}\label{mainteo1}
Let $M$ be an MCD-handledody, with a complete hyperbolic structure on its interior and whose conformal boundary is a closed Riemann surface $S$ of genus $g \geq 2$.
\begin{enumerate}[leftmargin=15pt]
\item If $g=2$, then $|{\rm Aut}(S;M)| \leq 12$. Equality (only) holds for ${\rm Aut}(S;M)$ isomorphic to the dihedral group of order $12$.  
\item If $g \geq 3$, then $|{\rm Aut}(S;M)|<12(g-1)$.
\end{enumerate}
\end{theo}

\section{Preliminaries}\label{Sec:prelim}
We will denote by the symbol $G<K$ (respectively, $G \lhd K$) to indicate that $G$ is a subgroup (respectively, a normal subgroup) of the group $K$.

\subsection{Kleinian and function groups}
In this section, we recall some of the generalities on Kleinian groups we will need in this paper (see, for instance, in the books \cite{Maskit:book,MT}). 
A {\it Kleinian group} is a discrete subgroup of the group  ${\mathbb M} \cong {\rm PSL}_{2}({\mathbb C})$  of M\"obius transformations. Each Kleinian group $K$ has associated its {\it region of discontinuity} $\Omega \subset \widehat{\mathbb C}$ (an open set, which might or not be empty, consisting of those points on which $K$ acts properly discontinuous) and its {\it limit set}
$\Lambda=\widehat{\mathbb C} \setminus \Omega$. Let us observe that if $G \leq K$ and $G$ have finite index in $K$, then both have the same region of discontinuity. 

\begin{rema}
The Riemann sphere $\widehat{\mathbb C}$ can be seen as the conformal boundary of the $3$-dimensional hyperbolic space ${\mathbb H}^{3}={\mathbb C} \times (0,+\infty)$ (with the complete Riemannian metric $ds^{2}=(|dz|^{2}+dt^{2})/t^{2}$); its group of orientation-preserving isometries can be identified with ${\rm PSL}_{2}({\mathbb C})$. A Kleinian group $K$ is, in this way, a discrete group of isometries of the hyperbolic $3$-space. The quotient $({\mathbb H}^{3} \cup \Omega)/K$ is a hyperbolic $3$-orbifold (a manifold if $K$ has no torsion), whose interior ${\mathbb H}^{3}/K$ has a complete hyperbolic metric and whose conformal boundary is $\Omega/K$. A Kleinian group is {\it geometrically finite} if it has a finite-sided polyhedron as a fundamental domain (in particular, it is a finitely generated group).
\end{rema}

A {\it function group} is a finitely generated Kleinian group with an invariant component in its region of discontinuity. Basic examples of function groups are the following ones.
 \begin{enumerate}[leftmargin=15pt] 
\item {\it Elementary groups}:  Kleinian groups with a finite limit set (so, the limit set of cardinality at most $2$).

\item {\it Quasifuchsian groups}: finitely generated Kleinian groups whose region of discontinuity has two invariant components. These groups are known to be quasiconformally conjugate to Fuchsian groups (of the first kind). 

\item {\it Totally degenerate groups}: finitely generated non-elementary Kleinian groups whose region of discontinuity is both connected and simply connected. 

\end{enumerate}

\begin{rema}
In \cite{Maskit:function1, Maskit:function}, Maskit proved that every function group can be constructed from elementary groups, quasifuchsian groups, and degenerate groups by a finite number of applications of the Klein-Maskit combination theorem \cite{Maskit:book, Maskit:CombIV}. If, in the construction, there are not totally degenerated groups, then the constructed group is geometrically finite \cite{Maskit:book}.
\end{rema}

\subsection{MDC-Schottky groups}
We define the trivial group as the {\it MDC-Schottky group of rank zero}.
An {\it MDC-Schottky group of rank one} is a double-cusped parabolic group, that is, a Kleinian group conjugated, by a suitable M\"obius transformation, to $G=\langle A(z)=z+1, B(z)=z+\tau\rangle \cong {\mathbb Z}^{2}$, for some complex number $\tau$ with positive imaginary part.  In this case, ${\mathbb C}/G$ is an MCD-handlebody of genus one.
An {\it MDC-Schottky group of rank $g \geq 2$}  is a Kleinian group $G$ obtained as the free product (using first Klein-Maskit's combination theorem \cite{Maskit:book, Maskit:CombIV}) of $g$ MDC-Schottky groups of rank one. It follows from the Klein-Maskit combination theorem that: (i) $G$ is a geometrically finite function group, (ii) its limit set is totally disconnected, (iii)
$M_{G}=({\mathbb H}^{3} \cup \Omega)/G$ is an MDC-handlebody of genus $g$ whose conformal boundary is the closed Riemann surface $S_{G}=\Omega/G$ of genus $g$, and (iv) 
every non-loxodromic transformation in $G$ is a parabolic transformation that belongs to a $G$-conjugate of one of the double-cusped parabolic subgroups.

In the above definition, MDC stands for ``maximal double-cusped".


\subsection{MDC-Schottky extension groups}\label{Sec:g=0,1}
A Kleinian group $K$ is called an {\it MDC-Schottky extension group} of rank $g$ if it contains an MDC-Schottky group $G$ of rank $g$ as a finite index normal subgroup. In particular, $K$ is a function group (with the same region of discontinuity as $G$), with a totally disconnected limit set, and such that each of its non-loxodromic transformations is either a finite order elliptic transformation or it is $K$-conjugated to a parabolic transformation belonging to a double-cusped parabolic factor subgroup of $G$. As a consequence of Maskit's decomposition theorem \cite{Maskit:function1, Maskit:function}, $K$ can be constructed (by use of the Klein-Maskit theorem) using finite groups of M\"obius transformations, double-cusped parabolic groups and cyclic groups generated by loxodromic transformations.
In Theorem \ref{estructura} (in Section \ref{Sec:main}), we provide a concrete decomposition structure of the MDC-Schottky extension groups in terms of the Klein-Maskit combination theorems.  Below, we describe those MDC-Schottky extension groups of ranks $0$ and $1$.

\subsubsection{\bf MDC-Schottky extension groups of rank zero}
MDC-Schottky extension groups of rank zero are the finite M\"obius groups, that is,  either (i) the trivial group, (ii) a cyclic group generated by a finite order elliptic transformation, (iii) a dihedral group $D_{n}$, (iv) the alternating group ${\mathcal A}_{4}$, (v) the alternating group ${\mathcal A}_{5}$ or (vi) the symmetric group ${\mathfrak S}_{4}$. 

\subsubsection{\bf MDC-Schottky extension groups of rank one}
An MDC-Schottky group $G$ of rank one is conjugated, by a suitable M\"obius transformation, to the double-parabolic Kleinian group $\langle A_{1}(z)=z+1, A_{2}(z)=z+ \tau \rangle \cong {\mathbb Z}^{2}$,   where $\tau$ is a complex number with positive imaginary part. In particular, $S={\mathbb C}/G$ is a genus one Riemann surface (i.e., a torus). If $K$ contains
$G$ as a finite index normal subgroup, then $K/G$ is a group of conformal automorphisms of $S$. These groups are described in the classical book of Burnside \cite{Burnside} (see also \cite{GT,HKK,Maskit:book}) and they are given, up to conjugation by a M\"obius transformation, by the following ones:
$$
\begin{array}{l}
K_{1}=\langle A(z)=z+1+\tau, \; B(z)=z+1 \rangle, \\ 
K_{2}=\langle A(z)=-z+1+\tau,B(z)=-z+1, \; C(z)=z+1 \rangle,\\
K_{3}=\langle A(z)=e^{2\pi i/3}z, \; B(z)=z+1\rangle,  \; (\tau=e^{2 \pi i/3})\\ 
K_{4}=\langle A(z)=iz, \; B(z)=-z+1\rangle, \; (\tau=i)\\
K_{6}=\langle A(z)=e^{\pi i/3}z, \; B(z)=-z+1\rangle, \; (\tau=e^{\pi i/3}) \\
K_{22}=\langle A(z)=-z+1+\tau, \; B(z)=-z+1, \; C(z)=-z+\tau/2, \; D(z)=-z+\tau/2+2 \rangle. 
\end{array}
$$

Note that :
\begin{enumerate}[leftmargin=15pt]
\item Both $K_{1}$ and $K_{3}$ are normalized by $j(z)=-z$. 

\item $K_{6}=\langle K_{3}, j\rangle$ and $K_{2}=\langle K_{1},j\rangle$. 

\item  $K_{2}$ is an index two subgroup of $K_{22}$.

\item The group $K_{6}$ uniformizes an orbifold with signature $(0,3;2,3,6)$, the group $K_{3}$ uniformizes an orbifold with signature  $(0,3;3,3,3)$, the group $K_{4}$ uniformizes an orbifold with signature $(0,3;2,4,4)$, the groups $K_{2}$ and $K_{22}$ both uniformize an orbifold with signature $(0,4;2,2,2,2)$ and $K_{1}$ uniformizes a torus.
\end{enumerate}

In the rest of this paper, we use the letter $K_{j}$ to denote any group obtained by a ${\rm PSL}_{2}({\mathbb C})$-conjugate of the above-normalized group $K_{j}$.


\subsection{MDC-Schottky groups and uniformizations}
An {\it uniformization} of a closed Riemann surface $S$ is a triple $(\Delta,G,P)$, where $G$ is a function group, $\Delta$ is a $G$-invariant connected component of its region of discontinuity and $P:\Delta \to S$ is a Galois planar covering with $G$ as its deck group.  The collection of uniformizations of $S$ is partially ordered: $(\Delta_{1},G_{1},P_{1}) \geq (\Delta_{2},G_{2},P_{2})$ if there is a (holomorphic) covering map $Q:\Delta_{1} \to \Delta_{2}$ such that $P_{1}=P_{2} \circ Q$. If $Q$ is an isomorphism, then these uniformizations are called {\it equivalent}.
Highest uniformizations are provided when $\Delta$ is simply connected (for instance, if $g \geq 2$, those are provided when $G$ is a co-compact Fuchsian group; called {\it Fuchsian uniformizations}). The lowest ones are those when $G$ is a {\it Schottky group} of rank $g$; called {\it Schottky uniformizations}.
 If $g \geq 1$, then among those uniformizations $(\Delta,G,P)$,  for which  $G$ is geometrically finite and $\Delta$ is all of its region of discontinuity, the highest ones occur when $G$ is an MDC-Schottky group of rank $g$; called  {\it MDC-Schottky uniformizations}.

Planarity theorem \cite{Maskit:book} states that for a Galois planar covering $P:\Delta \to S$, of a closed Riemann surface $S$, there is a finite collection of pairwise disjoint simple loops on $S$ so that the covering $P$ is the highest Galois planar covering of $S$ for which these loops (with the correct powers) lift to loops. The next result describes such a collection of loops providing the Galois planar coverings of $S$ with deck group an MDC-Schottky group.

\begin{lemm}\label{lema1}
Let $S$ be a closed Riemann surface of genus $g \geq 1$.
Any collection of pairwise disjoint loops in $S$, defining a Galois planar covering of  $S$ with deck group an MDC-Schottky group, consists of dividing simple loops so that all of them divide $S$ into surfaces of genus at most $1$.
\end{lemm}
\begin{proof}
As an MDC-Schottky group is torsion-free, any loop in the family is a simple loop (that is, not the non-trivial positive power of a simple loop). If some of the connected components of the complement of such a set of defining loops have genus at least $2$, then the covering group should contain a quasifuchsian group of the first kind; a contradiction to the fact that a MDC-Schottky group has a totally disconnected limit set.
\end{proof}

\subsection{Lifting automorphisms to MDC-Schottky uniformizations}
Let $(\Omega,G,P)$ be an MDC-Schottky uniformization of the closed Riemann surface $S$. 
We say that a group $H$, of conformal automorphisms of $S$,  {\it lifts with respect} to this uniformization if, for every $h \in H$, there is a M\"obius transformation $\widehat{h}$ normalizing $G$ so that $h P=P \widehat{h}$. Note that, the group $K$ of these all lifted automorphisms is an MCD-Schottky extension group.
Below is a necessary and sufficient condition for $H$ to lift to an MDC-Schottky uniformization. 

\begin{theo}\label{loops}
Let $(\Omega,G,P)$ be a MDC-Schottky uniformization of $S$ and 
let $H$ be a (necessarily finite) group of conformal automorphisms of $S$. Then $H$ lifts if and only if there is a system of pairwise disjoint simple loops ${\mathcal F} \subset S$ so that:
\begin{itemize}[leftmargin=15pt]
\item[(i)] ${\mathcal F}$ is $H$-invariant;
\item[(ii)] each loop in ${\mathcal F}$ is a homotopically non-trivial dividing simple loop;
\item[(iii)] if $g=2$, then ${\mathcal F}$ consists of exactly one dividing loop, that divides $S$ into two tori;
\item[(iv)] if $g \geq 3$, then either
\begin{itemize}
\item[(iv.1)] $S \setminus {\mathcal F}$ is a collection of surfaces of genus $1$ and an extra one of genus $\gamma \in \{0,1\}$ which is $H$-invariant, or
\item[(iv.2)] ${\mathcal F}$ has an $H$-invariant loop and $S \setminus {\mathcal F}$ is a collection of surfaces of genus $1$;
\end{itemize}
and
\item[(v)] ${\mathcal F}$ defines the Galois  holomorphic planar covering $P:\Omega \to S$ with covering group $G$.
\end{itemize}
\end{theo}
\begin{proof}
It follows, from the general theory on covering spaces, that the existence of a system of loops as stated will be sufficient for $H$ to lift as a group of conformal automorphisms of $\Omega$. 
In this case, the liftings of the elements of $H$ provide a group $K$, of conformal automorphisms of $\Omega$, containing $G$ as a finite index normal subgroup such that $K/G=H$.
As $\Omega$ is a domain of class $O_{AD}$ \cite{AS}, it follows that each conformal automorphism of $\Omega$ is the restriction of a M\"obius transformation, so each lifting is a M\"obius transformation. 

In the other direction, let us assume the group $H$ lifts, and let $K$ be the Kleinian group one obtains by the lifting of all the elements of $H$. So $G$ is a finite index normal subgroup of $K$ and $K/G=H$. By Lemma \ref{lema1}, there is a simple loop on $S$ dividing it into two surfaces, one of genus $1$ and the other of genus $g-1$, and which lifts (under $P$) to simple loops. We consider a shortest one (in the natural hyperbolic metric on $S$), say $\delta_{1} \subset S$.
By the minimum hyperbolic length property of $\delta_{1}$ and the planarity of $\Omega$, for each $h \in H$ either $h(\delta_{1})=\delta_{1}$ or $h(\delta_{1}) \cap \delta_{1}=\emptyset$. Also, each $h(\delta_{1})$ is a simple closed geodesic, of same length as $\delta_{1}$, that divides $S$ into $2$ surfaces, one of them of genus $1$. 
If $g=2$, then we are done as such a loop $\delta_{1}$ is necessarily $H$-invariant.
Let us now assume $g \geq 3$. Consider the collection ${\mathcal F}_{1}$ of simple loops obtained as orbit in $H$ of $\delta_{1}$. Clearly, ${\mathcal F}_{1}$ divides $S$ into an $H$-invariant collection of genus $1$ surfaces and one $H$-invariant surface $S_{1}$. If the genus of $S_{1}$ is either $0,1$, then we are done. Assume now on that the genus of $S_{1}$ is  $\gamma \geq 2$. 
Let $\Omega_{1} \subset \Omega$ be a connected component of $\Omega \setminus P^{-1}({\mathcal F}_{1})$
 such that $P(\Omega_{1})=S_{1}$. As each of the loops in ${\mathcal F}_{1}$ lift to a loop under $P$ and that such a family is $H$-invariant, the $G$-stabilizer of $\Omega_{1}$ is a MDC-Schottky group $G_{1}$ of rank $\gamma$ and its $K$-stabilizer is a Kleinian group $\Gamma_{1}<K$ so that $G_{1} \lhd \Gamma_{1}$ and $H=\Gamma_{1}/G_{1}$. By Lemma \ref{lema1}, we can work with this new surface $S_{1}$ (once we glue one disc to each border of it) as done with the surface $S$. This procedure will end after a finite number of steps to get the desired simple loops.
\end{proof}


\begin{rema}\label{Obs6}
(1) Associated to the collection ${\mathcal F}$ of Theorem \ref{loops} there is a finite tree. The vertices are the components of $S \setminus {\mathcal F}$ and the edges are given by each loop in ${\mathcal F}$. The edge, determined by a loop $\delta \in {\mathcal F}$, connects the two vertices corresponding to the two components of $S \setminus {\mathcal F}$ having $\delta$ as a common boundary loop. For instance, for $g=2$, the tree has exactly one edge and two vertices. (2) In Part (iv.2) of Theorem \ref{loops}, the $H$-invariant loop $\delta \in {\mathcal F}$ is the common boundary of two tori, say $T_{1}$ and $T_{2}$ (components of $S \setminus {\mathcal F}$). If we consider $T_{1} \cup \delta \cup T_{2}$, then we obtain a genus two $H$-invariant surface. In this case, the collection ${\mathcal F}^{*}={\mathcal F} \setminus \{\delta\}$ still $H$-invariant and $S \setminus {\mathcal F}^{*}$ consists of one $H$-invariant surface of genus two and a collection of genus one surfaces which are permuted between them under the action of $H$. But this new collection of loops does not determine the covering $P$.
\end{rema}

\subsection{MCD-Schottky structures}
We have seen that, if $G$ is an MCD-Schottky group of rank $g \geq 1$, then $M_{G}$ is an MCD-handlebody of genus $g$ whose conformal boundary is a closed Riemann surface of genus $g$.
Due to the Marden conjecture (or tame ends conjecture) proved by Agol \cite{Agol} and Calegari-Gabai \cite{Calegari}, each Kleinian group isomorphic to a free product of $g$ copies of ${\mathbb Z}^{2}$ provides a hyperbolic structure on the interior $M^{0}$ of an MCD-handlebody $M$ of genus $g \geq 1$. For each of these hyperbolic structures, the corresponding conformal boundary is a subset of the topological boundary of $M$.

\begin{prop}
Every hyperbolic structure on the interior of an MCD-handlebody of genus $g \geq 1$, whose conformal boundary is a closed Riemann surface of genus $g$, is provided by an MCD-Schottky group of rank $g$.
\end{prop}
\begin{proof}
Let $M$ be an MDC-handlebody of genus $g \geq 1$.
A hyperbolic structure on the interior $M^{0}$ of $M$ is provided by a Kleinian group $\Gamma$ so that ${\mathbb H}^{3}/\Gamma$ is homeomorphic to $M^{0}$. It follows that such a Kleinian group $\Gamma$ is isomorphic to a free product of $g$ copies of ${\mathbb Z}^{2}$. 
Each of the cusps on $M^{0}$ provides a rank two cusp of such a hyperbolic structure. As the conformal boundary is assumed to be all the topological boundary,  the group $\Gamma$ is necessarily an MDC-Schottky group of rank $g$.
\end{proof}

We will say that an MCD-Schottky group produces a {\it MDC-Schottky structure} on $M$.

\section{Main Results}\label{Sec:main}
The strategy used in \cite{Maskit:function1, Maskit:function,Maskit:decomposition}, to provide the structural description of a function group $K$, was to construct a collection of structure loops and structure regions (inside the region of discontinuity of $K$) and to use their $K$-stabilizers together with Klein-Maskit's combination theorems \cite{Maskit:book, Maskit:CombIV}. In the case that $K$ is an MCD-Schottky extension group, we will proceed to make such a structural description explicit (Theorem \ref{estructura}).

\subsection{Structural description of MDC-Schottky extension groups}
The decomposition structure of the MDC-Schottky extension groups of rank $g \geq 2$ is given by the following (its proof is provided in Section \ref{Sec:prueba1}).

\begin{theo}[Geometrical description of MDC-Schottky extension groups]\label{estructura}
Every MDC-Schottky extension group is obtained from the Klein-Maskit combination theorems as the free amalgamated product
$$
\begin{array}{c}
K=G*K^{0}*_{\langle t_{1} \rangle} K^{1} *_{\langle t_{3} \rangle} K^{3} *_{\langle t_{4} \rangle} \cdots *_{\langle t_{n} \rangle} K^{n},
\end{array}
$$
where 
\begin{enumerate}[leftmargin=15pt]
\item $G$ is a MDC-Schottky group;
\item  $K^{0}$ is either trivial or a finite group of M\"obius transformations or one of the MDC-Schottky extension groups of rank $1$, and 
\item each $K^{j}$ ($j\geq 1$) is a MDC-Schottky extension groups of rank $1$,
\item $\langle t_{1}\rangle=(G*K^{0}) \cap K^{1}$, 
 \item for $j=2,\ldots,n$, $\langle t_{j}\rangle=(G*K^{0}*_{\langle t_{1} \rangle} K^{1}*_{\langle t_{2} \rangle} \cdots *_{\langle t_{j-1} \rangle} K^{j-1})\cap K^{j}$, 
 \item in each case, $t_{j}$ is either the identity or an elliptic transformation of order belonging to $\{2,3,4,6\}$.
\end{enumerate}
\end{theo}

\subsection{An application to automorphisms of MDC-Schottky handlebodies}
A {\it conformal automorphism} of an MDC-handlebody $M$, with a fixed MDC-Schottky structure induced by an MCD-Schottky group $G$, is an orientation preserving self-homeomorphism of $M$ whose restriction to the interior is an isometry for the corresponding hyperbolic structure (or equivalently, its restriction to the conformal boundary surface is conformal). We denote by ${\rm Aut}^{+}(M)$ the group of conformal automorphisms of $M$.  
Note that, each MDC-Schottky extension group $K$, containing $G$ as a finite index normal subgroup, produces a finite order subgroup of ${\rm Aut}^{+}(M)$ and, reciprocally, all such subgroups of conformal automorphisms are obtained in this way.
The above classification of the MDC-Schottky extension groups (Theorem \ref{estructura}), permits us to provide the following classification of the groups of automorphisms of an MDC-handlebody with an MDC-Schottky structure and the possible quotient orbifolds (its proof is provided in Section \ref{Sec:prueba2}).

\begin{theo}\label{cor1}
Let $M$ be an MDC-handlebody of genus $g \geq 2$, with an MDC-Schottky structure, and let $H<{\rm Aut}^{+}(M)$ be non-trivial. Then exactly one of the following situations holds.
\begin{itemize}[leftmargin=15pt]
\item[(1)] $g=2$ and there is a compressible disc $D \subset M$, invariant under $H$, and $M \setminus D$ consists of two components, $M_{1}$ and $M_{2}$ (these are MDC-handlebodies of genus $1$). The group $H$ is is either 
\begin{enumerate}
\item[(1.1)] a cyclic group of order $2$, $3$, $4$ or $6$,  keeping invariant $M_{1}$ or 
\item[(1.2)] a cyclic group of oder two that interchanges $M_{1}$ with $M_{2}$ or 
\item[(1.3)] a dihedral group of order either $4$, $6$, $8$ or $12$ (containing one involution that permutes $M_{1}$ with $M_{2}$ and the maximal cyclic group preserves $M_{1}$).  
\end{enumerate}

In particular, ${\rm Aut}^{+}(M)$ is isomorphic to either: (i) ${\mathbb Z}_{2}$ or (ii) $D_{n}$, where $n \in \{2,4,6\}$. In case (i), the generic case, each of the two cusps is invariant, and in case  (ii) one involution permutes both cups.

\item[(2)]  $g \geq 3$ and there is a compressible disc $D \subset M$, invariant under $H$, and $M \setminus D$ consists of two components, $M_{1}$ and $M_{2}$ (these are MDC-handlebodies whose sum of genus equals to $g$). The group $H$ is either  
\begin{enumerate}
\item[(2.1)] a cyclic group of order $2$, $3$, $4$ or $6$, keeping invariant $M_{1}$ or 
\item[(2.2)] a cyclic group of order two that interchanges  $M_{1}$ with $M_{2}$ or 
\item[(2.3)] a dihedral group of order either $4$, $6$, $8$ or $12$, containing one involution that permutes $M_{1}$ with $M_{2}$ and the maximal cyclic group preserves $M_{1}$. 
\end{enumerate}

\item[(3)] $g \geq 3$, there is not  a compressible disc $D \subset M$ as in (2) above and there is a $H$-invariant MDC-handlebody $N \subset M$ of genus $\gamma \in \{0,1\}$.
\begin{itemize}
\item[(3.1)] If $\gamma=0$, then $H$ is a finite group isomorphic to either a cyclic group or a dihedral group or ${\mathcal A}_{4}$ or ${\mathcal A}_{5}$ or ${\mathfrak S}_{4}$.

\item[(3.2)] If $\gamma=1$, then $H$ is a finite group isomorphic to a finite group of 
automorphisms of a genus one surface.
\end{itemize}
\end{itemize}
\end{theo}

\subsection{Upper bounds for MCD-Schottky extension groups}

Let $K$ be an MDC-Schottky extension group containing, as a finite index normal subgroup, an MDC-Schottky group $G$ of rank $g \geq 2$. 
Let $\Omega$ be the region of discontinuity of $K$. Since $G$ has finite index in $K$, then $\Omega$ is also the region of discontinuity of $G$. 
As $H=K/G$ is a group of conformal automorphisms of $S=\Omega/G$, then $[K:G]=|H| \leq 84(g-1)$ by the Hurwitz bound.
Now, the only possibility for $[K:G]> 12(g-1)$ is for $S/H=\Omega/K$ to have signature of type $(0;a,b,c)$. But a geometrically finite function group providing a uniformization of an orbifold with such kind of signature should be a Fuchsian group of the first kind \cite{Kra}, a contradiction to the fact that the limit set of $K$, the same as for $G$, is a Cantor set. It follows then that $[K:G] \leq 12(g-1)$.

\begin{coro}\label{maxteo}
Let $K$ be an MDC-Schottky extension group containing an MDC-Schottky group $G$ of rank $g \geq 2$ as a finite index normal subgroup. 
\begin{enumerate}[leftmargin=15pt]
\item[(i)] If $g \geq 3$,  then $[K:G]<12(g-1)$.
\item[(ii)] If $g=2$, then $[K:G] \leq 12$ and, moreover, if $[K:G]=12$, then $K/G$ is isomorphic to the dihedral group of order $12$. 
\end{enumerate}
\end{coro}

In terms of MDC-handlebodies, the above can be written as follows (this is the statement of Theorem \ref{mainteo1} in the introduction).

\begin{coro}\label{maxteo1}
Let $M$ be an MDC-handlebody of genus $g \geq 2$ with an MDC-Schottky structure.
If $ g \geq 3$, then $|{\rm Aut}^{+}(M)| < 12(g-1)$. If $g=2$, then $|{\rm Aut}^{+}(M)| \leq 12$ and if 
$|{\rm Aut}^{+}(M)|=12$, then ${\rm Aut}^{+}(M)$ is the dihedral group of order $12$.
\end{coro}

\begin{rema}[A maximal example for $g=2$]\label{ejemplo1}
In the above corollary, it is stated that there is an MDC-Schottky extension group containing an MDC-Schottky group of rank two as a normal subgroup of index $12$. An example of this situation is provided below.
Let $\tau=e^{\pi i/3}$, $A(z)=z+1$, $B(z)=z+\tau$ and $E(z)=\tau z$. Choose a circle $D$ with center at $0$ and radius $1/4$ and let $p_{1}, p_{2} \in D$ be the intersection points of $D$ with the line through $0$ and $1+\tau$. Let $F$ be the conformal involution with fixed points $p_{1}$ and $p_{2}$.  Set $K=\langle A,B,E,F\rangle$ and $G=\langle A,B\rangle * \langle FAF,FBF\rangle$. In this case, (i) $G \lhd K$, (ii) $K/G \cong D_{6}$ (the dihedral group of order $12$) and (iii) $G$ is an MDC-Schottky group of rank $2$. So, $K$ is an MDC-Schottky extension group of rank $2$ as desired.
\end{rema}

\section{Proof of Theorem \ref{estructura}}\label{Sec:prueba1}
Let $K$ be an MDC-Schottky extension group of rank $g \geq 2$, with $\Omega$ as its region of discontinuity. By the definition, there is an MDC-Schottky group of rank $g$
$G \lhd K$. In this case, the region of discontinuity of $G$ is also $\Omega$. Set $S=\Omega/G$ and let $P:\Omega \to S=\Omega/G$ be the natural holomorphic Galois planar covering induced by $G$. Let $H=K/G$ be the induced group of conformal automorphisms of $S$. 
As, in this case, the group $H$ lifts,
Theorem \ref{loops} provides us with a system of loops  ${\mathcal F}$ on $S$ which is invariant under the action of $H$ and satisfies all the properties in that theorem stated.
Let ${\mathfrak T}$ be the finite tree induced by ${\mathcal F}$ (as seen in (1) of Remark \ref{Obs6}).
Each loop in $\mathcal F$ lifts, under $P$, to a collection of pairwise disjoint simple loops, each one with trivial $G$-stabilizer (since $G$ is torsion-free); called an {\it structure loop}. Let us denote by ${\mathcal G}$ the collection of all such structure loops.
Each of the connected components of $\Omega \setminus {\mathcal G}$ is called a {\it structure region}.

\begin{rema}
Each loop in ${\mathcal F}$ bounds a compressible disc in $M_{G}=({\mathbb H}^{3} \cup \Omega)/G$, and such a family of discs are pairwise disjoint, each one with border exactly one of these loops. These discs form a collection that is $H$-invariant and they cut-off $M_{G}$ into a collection of $g$ MDC-handlebodies of genus $1$ and maybe a $3$-ball.
If there is such a ball, then it is $H$-invariant. If there is no such a ball, then either (i) one of the MDC-handlebodies of genus $1$ is $H$ invariant or (ii) none of them is $H$-invariant, but there are two of them which are adjacent (that is, they have the same disc on the border) and the union is $H$-invariant. 
Each of the structure loops is the boundary of a {\it structure disc} in ${\mathbb H}^{3}$ so that all these structure discs form a $K$-invariant set. These structure discs divide ${\mathbb H}^{3}$ into {\it structure $3$-regions}; they contain in their boundary the structure regions. This, together with the finite tree ${\mathfrak T}$, provides a first picture for the decomposition of the action of $H$ in $M$ and the picture of the structural form of $K$.
\end{rema}

\begin{lemm}\label{lema5-1}
if $\delta \in {\mathcal G}$ be a structure loop and $R$ is a structure region, then both restrictions $P:\delta \to P(\delta)$ and $P:R \to P(R)$ are homeomorphism. 
\end{lemm}
\begin{proof}
By the construction, both the collection of structure loops and the collection of structure regions are $K$-invariant. Moreover, the $G$-stabilizer of either a structure loop or a structure region is a finite subgroup of $G$, and as $G$ is torsion-free, it is trivial. 
\end{proof}

By the conditions on the set of loops ${\mathcal F}$ given in Theorem \ref{loops}, if $R$ is a structure region, then either $P(R)$ is a genus one bordered surface or $P(R)$ is a genus zero $H$-invariant surface.  As a consequence of Lemma \ref{lema5-1}, we observe the following.

\begin{lemm}\label{lema5-1.1}
The $K$-stabilizer of a structure loop $\delta \in {\mathcal G}$ is isomorphic to the $H$-stabilizer of $P(\delta) \in {\mathcal F}$. The $K$-stabilizer of a structure region $R$ so that $P(R)$ is a genus one surface is an MDC-Schottky extension group of rank $1$. If $R$ is a structure region so that $P(R)$ is a genus zero surface, then its $K$-stabilizer is isomorphic to $H$.
\end{lemm}


\subsection{Case ${\bf g=2}$}\label{Sec:g>2-1}
In this case ${\mathcal F}=\{\alpha\}$ and $\alpha$ is invariant under $H$. Let $S_{1}$ and $S_{2}$ be the two bordered tori  in $S \setminus \{\alpha\}$ and let $H_{S_{j}}$ be the $H$-stabilizer of $S_{j}$, for $j=1,2$.

\begin{lemm}\label{lema5-2}
$H$ is either (i) trivial or (ii) a cyclic group of order $2$ that permutes both $S_{1}$ and $S_{2}$ or (iii) a cyclic group of order $d \in \{2,3,4,6\}$ that preserves each of the two $S_{1}$ and $S_{2}$ or (iv) a dihedral group of order $r \in \{2,4,6,8,12\}$ generated by an involution permuting $S_{1}$ with $S_{2}$ and a cyclic group that preserves each $S_{j}$.  In cases (i) and (ii) the groups $H_{S_{j}}$ are trivial, in case (iii) $H=H_{S_{j}}$ and in case (iv) $H_{S_{j}}$ is the index two cyclic subgroup of $H$ of order $r/2$.
\end{lemm}
\begin{proof}
Let us assume first that $H$ does not contain an element that permutes $S_{1}$ with $S_{2}$. In that situation, $H$ is a finite group of conformal automorphisms of $S_{1}$ keeping invariant the loop $\alpha$. It follows that there is a torus (by gluing a disc to $S_{1}$ along the loop $\alpha$) admitting $H$ as a group of conformal automorphisms fixing a point. Then $H$ should be a finite cyclic group. As the non-trivial finite cyclic groups acting on genus one with fixed points are known to be of order $d \in \{2,3,4,6\}$, we are done in this case. If there is some $j \in H$ that permutes $S_{1}$ with $S_{2}$, then $j^{2}$ keeps invariant each $S_{j}$. Let $H^{*}$ be the index two subgroup of $H$ keeping invariant $S_{1}$. By the previous, $H^{*}$ is either the trivial group or a cyclic group of order $d \in \{2,3,4,6\}$. If $H^{*}$ is trivial, then $j^{2}=1$ and $H$ is a cyclic group of order two generated by an involution that permutes $S_{1}$ with $S_{2}$.
Let us now assume $H^{*}$ is non-trivial, say generated by $h$. 
As $j$ is orientation preserving in $S$, keeps invariant $\alpha$ and interchanges $S_{1}$ with $S_{2}$, it must interchanges the orientation of $\alpha$ and $j^{2}=1$. It follows, since $h$ preserves the orientation of $\alpha$, that $jhj^{-1}=h^{-1}$, in particular, $H$ is a dihedral group as desired.
\end{proof}

Let $\delta \in {\mathcal G}$ be any structure loop and let 
$R_{1}$ and $R_{2}$ be the two structure regions having $\delta$ at the border.
We may assume $P(R_{1})=S_{1}$ and  $P(R_{2})=S_{2}$. 
Let $K_{\delta}$ be the $K$-stabilizer of $\delta$ and let $K_{R_{j}}$ be the $K$-stabilizer of $R_{j}$, for $j=1,2$. By Lemma \ref{lema5-1} $K_{\delta} \cong H$ and $K_{R_{j}}$ is a MDC-Schottky extension group of rank $1$ containing a MDC-Schottky group of rank $1$, say $G_{j}$, so that $K_{R_{j}}/G_{j}=H_{S_{j}}$. By Lemma \ref{lema5-2}, one of the following holds.
\begin{enumerate}[leftmargin=15pt]
\item If $K_{\delta}$ is trivial, then $K=K_{R_{1}}*K_{R_{2}}$, that is, a MDC-Schottky group of rank $2$.
\item If $K_{\delta}$ is generated by an elliptic transformation of order $d \in \{2,3,4,6\}$ (keeping invariant each $R_{j}$), then both $K_{R_{1}}$ and $K_{R_{2}}$ are conjugate to the same MDC-Schottky extension group of rank $1$ and $K=K_{R_{1}}*_{K_{\delta}}K_{R_{2}}$
(see the left Figure \ref{figura1(ii)} for the case $K_{R_{j}}$ conjugate to $K_{d}$, $d=2,3,4,6$). The orbifold uniformized by $K$ has signature (a) $(0,6;2,2,2,2,2,2)$ if $d=2$, (b) $(0,4;3,3,3,3)$ if $d=3$, (c) $(0,4;2,2,4,4)$ if $d=4$ and (d) $(0,4;2,2,3,3)$ if $d=6$. 
\item If $K_{\delta}$ is generated by  a conformal involution that interchanges $R_{1}$ with $R_{2}$, then $K=K_{R_{1}}*K_{\delta}$ (see the middle Figure \ref{figura1(iii)}). The group $K$ uniformizes an orbifold with signature $(1,2;2,2)$.
\item If $K_{\delta}$ is a dihedral group generated by an elliptic transformation of order $d \in \{2,3,4,6\}$ and a conformal involution that permutes both structure regions $R_{1}$ and $R_{2}$, then $K=K_{R_{1}}*_{L}K_{\delta}$, where $L<K_{\delta}$ is the cyclic group of order $d$ (see the right
Figure \ref{figura1(iv)}). The orbifold uniformized by $K$ has signature (a) $(0,5;2,2,2,2,2)$ if $d=2$, (b) $(0,4;3,3,2,2)$ if $d=3$, (c) $(0,4;2,2,2,4)$ if $d=4$ and (d) $(0,4;2,2,2,3)$ if $d=6$. 
\end{enumerate}




\begin{figure}[htb]
    \begin{minipage}[t]{.3\textwidth}
        \centering
        \includegraphics[width=\textwidth]{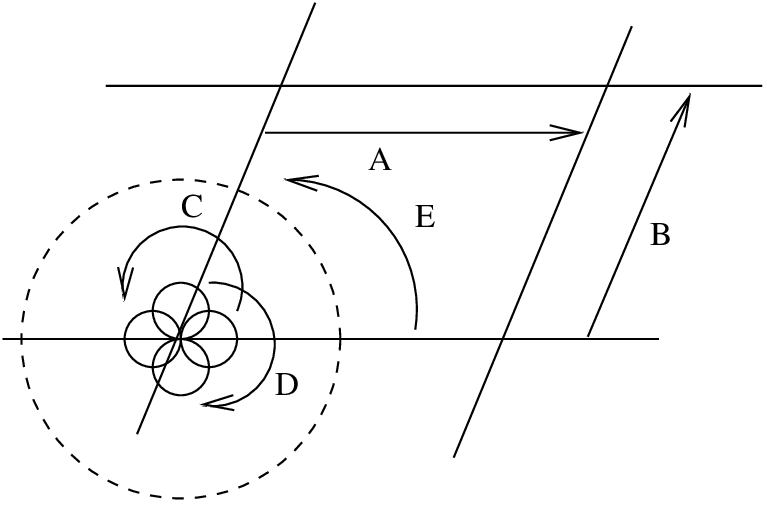}
       \subcaption{\tiny $K=\langle A,B, E\rangle *_{\langle E \rangle}\langle E,C,D\rangle$, $E^{d}=1$}
        \label{figura1(ii)}
    \end{minipage}
    \hfill
    \begin{minipage}[t]{.3\textwidth}
        \centering
        \includegraphics[width=\textwidth]{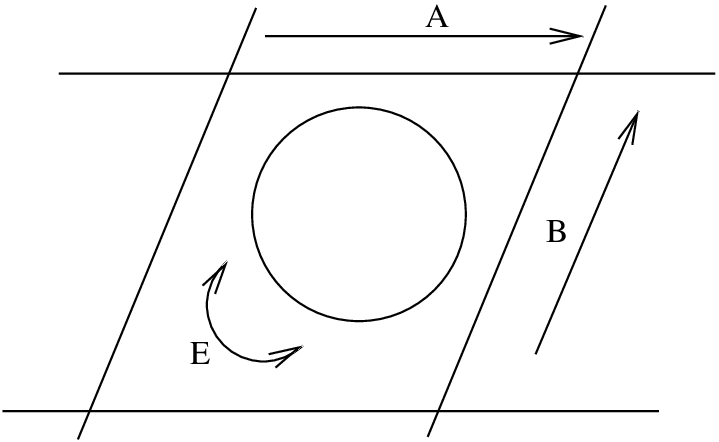}
        \subcaption{\tiny $K=\langle A,B,E\rangle$, $E^{2}=1$}
        \label{figura1(iii)}
    \end{minipage}
    \hfill
    \begin{minipage}[t]{.3\textwidth}
        \centering
        \includegraphics[width=\textwidth]{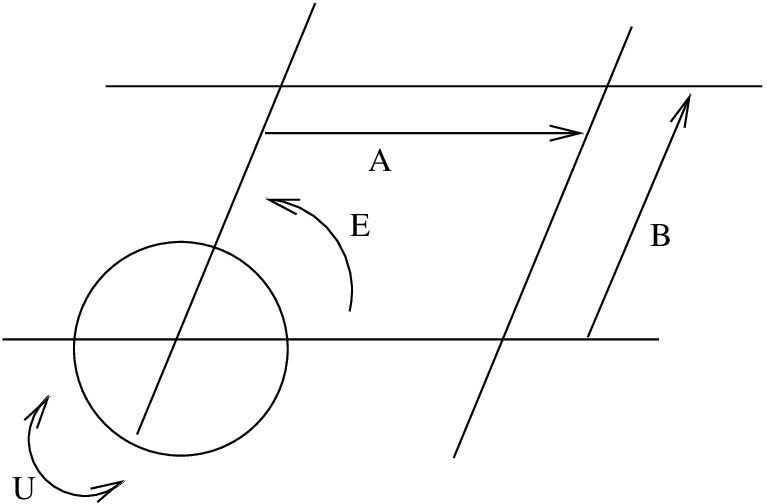}
        \subcaption{\tiny $K=\langle A, B,E,U\rangle$, $U^{2}=1=E^{d}$}
        \label{figura1(iv)}
    \end{minipage}  
   \caption{The groups $K$}
\end{figure}

\subsection{Case ${\bf g \geq 3}$}\label{Sec:g>2-2}
We divide the proof into three cases, depending on the existence of an $H$-invariant region in $S \setminus {\mathcal F}$.
\subsubsection{First case}
Let us assume there is a $H$-invariant genus zero surface $\widehat{R}_{0} \subset S \setminus {\mathcal F}$ and let us denote its boundary loops by $\alpha_{1},\ldots,\alpha_{r}$. In this case the finite tree  ${\mathfrak T}$ has a central vertex, associated to $\widehat{R}_{0}$, with exactly $r$ edges containing such a central vertex.

As $H$ stabilizes the planar surface $\widehat{R}_{0}$, it follows that $H$ is either (i) the trivial group or (ii) a finite cyclic group or (iii) ${\mathcal A}_{4}$ or (iv) ${\mathcal A}_{5}$ or(v)  ${\mathfrak S}_{4}$. Also, the $H$-invariance property of $\widehat{R}_{0}$ asserts that the $H$-stabilizer of any $\alpha_{j} \in {\mathcal F}$ is either trivial or a finite cyclic group which cannot interchange the two structure components about $\alpha_{j}$ (otherwise, there will be an element of $H$ that does not preserves $\widehat{R}_{0}$).

If $m=1,\ldots,r$, then we denote by $\widehat{R}_{m} \subset S \setminus \{\alpha_{m}\}$ the surface (different from $\widehat{R}_{0}$) of genus $g_{m} \geq 1$ containing $\alpha_{m}$ on its border; we denote by $H_{m}$ its $H$-stabilizer.

\begin{lemm}\label{lema2}
$H_{m}$ is also the $H$-stabilizer of $\alpha_{m}$ and it is either the trivial group or a cyclic group of order $d \in \{2,3,4,6\}$.
\end{lemm}
\begin{proof}
As $H$ stabilizes $\widehat{R}_{0}$, it follows that $H_{m}$ also stabilizes $\alpha_{m}$. Now, as no element of $H$ can interchange $\widehat{R}_{m}$ with $\widehat{R}_{0}$ (one has a positive genus and the other is of genus zero), it holds that the $H$-stabilizers of $\widehat{R}_{m}$ and $\alpha_{m}$ are the same.
Now, let $T \subset {\widehat R}_{m}$ be the connected component of $S \setminus {\mathcal F}$ that contains $\alpha_{m}$ in its border. We know that $T$ is a genus one surface. As $H_{m}$ should also stabilize $T$, it extends to act as a group of conformal automorphisms on a closed Riemann surface of genus $1$ fixing one point (corresponding to the loop $\alpha$). It follows that $H_{m}$ is either the trivial group or a finite cyclic group of the required order.
\end{proof}

\begin{lemm}\label{lema3}
Let $T \subset S \setminus {\mathcal F}$ be different from $\widehat{R}_{0}$. Then the $H$-stabilizer of $T$ is a subgroup of some of the cyclic groups $H_{m}$ and the $H$-stabilizer of any of its boundary loops is also a subgroup of such subgroup.
\end{lemm}
\begin{proof}
Let us consider a connected component $T \subset S \setminus {\mathcal F}$, different from $\widehat{R}_{0}$, so
it is a genus one surface with some finite number of boundary loops. The component $T$ must be contained inside some $\widehat{R}_{m}$, for $m=1,\ldots,r$. It follows from Lemma \ref{lema2} that the $H$-stabilizer of $T$ is a subgroup of $H_{m}$. If $\beta \in {\mathcal F}$ is a boundary loop of $T$, then its $H$-stabilizer is a subgroup of the $H$-stabilizer of $T$; otherwise, there will be an element of $H$ that not keeps invariant $\widehat{R}_{0}$.
\end{proof}

Each structure region is either a planar surface homeomorphic, under $P:\Omega \to S$, to the planar surface $\widehat{R}_{0}$ (we say a structure region of type (P)) or, as a consequence of Lemma \ref{lema3}, an infinitely bordered planar surface (called of type (T)) which is invariant under an MDC-Schottky extension group of rank $1$ (either being a MDC-Schottky group of rank $1$ or being a cyclic extension of it).

\begin{lemm}\label{lema4}
If  $\delta \in {\mathcal G}$ is a structure loop and $R_{1}$ and $R_{2}$ the two structure regions having $\delta$ on the border, then the $K$-stabilizer of $\delta$ is the intersection of the $K$-stabilizes of both $R_{j}$. 
\end{lemm}
\begin{proof}
If one of the structure regions is of type (P), then the other should be of type (T) and the result follows from Lemma \ref{lema2}. If both structure regions are of type (T), then it follows from Lemma \ref{lema3} that 
the $K$-stabilizer of $\delta$ is contained in the intersection of the $K$-stabilizers of $R_{1}$ and $R_{2}$. As each of the loops in ${\mathcal F}$ is a dividing loop in $S$ and $\widehat{R}_{0}$ is $H$-invariant, it follows that  intersection of the $K$-stabilizers of $R_{1}$ and $R_{2}$ stabilizes $\delta$. 
\end{proof}

Let us fix a structure region $R_{0}$ of type (P) and let us denote its boundary structure loops as $\delta_{1}$,..., $\delta_{r}$. Clearly, the $G$-stabilizer of $R_{0}$ is trivial and its $K$-stabilizer is a finite M\"obius group $K^{0}\cong H$. 

Denote by $R_{m}$, for each $m=1,\ldots,r$, the region in $\Omega$ containing $\delta_{m}$ as a border loop and so that $P(R_{m})=\widehat{R}_{m}$. The $G$-stabilizer of $R_{m}$ is an MDC-Schottky group $G^{m}$ of rank $g_{m} \geq 1$ (the genus of $\widehat{R}_{m}$) and, by Lemma \ref{lema2}, its $K$-stabilizer is an MDC-Schotky extension group $K^{m}$ so that $K^{m}/G^{M}=H_{m}$ (which is either trivial or a cyclic group of order $2$ or $3$ or $4$ or $6$).

Let us consider a maximal collection of structure loops in the border of $R_{0}$ which are pairwise non-equivalent under the action of $K^{0}$. Let us assume they are given by (after renumbering) $\delta_{1},\ldots,\delta_{s}$.

It follows, from Lemma \ref{lema4}, that $K$ will be obtained from the Klein-Maskit's combination theorems as an amalgamated free product of the groups $K^{0}$, $K^{1},.\ldots, K^{s}$, that is,
$$K=K^{0} *_{L_{1}}K^{1}*_{L_{2}}K^{2}*_{L_{3}} \cdots *_{L_{s}} K^{s},$$
where $L_{j}=K^{0} \cap K^{j}$. Let us recall that $L_{j}$, the $K$-stabilizer of $\delta_{j}$, is either trivial or a cyclic group of order $d \in \{2,3,4,6\}$. Note that if $\delta_{j}$  has trivial $K$-stabilizer, that is $L_{j}$ is trivial, then necessarilly $K_{j}=G^{j}$. Let us assume, after reordering, that $\delta_{1},\ldots, \delta_{u}$ are exactly the above structure loops with
trivial $K$-stabilizer. It follows that 
$$K=U*K^{0}*_{L_{u+1}}K^{u+1}*_{L_{u+2}}K^{u+2}*_{L_{u+3}} \cdots *_{L_{s}} K^{s},$$
where $U$ is an MDC-Schottky group. It follows that if $u=s$, then $K=U*K_{0}$ and we are done.

Let us now assume $u<s$. Let us denote by $K^{*}$ any one of the group $K^{m}$, for $m \in \{u+1,...,s\}$.
It follows from Lemma \ref{lema3} and Lemma \ref{lema4} that 
$$K^{*}=V_{1}*_{W_{1}}V_{2}*_{W_{2}}* \cdots *_{W_{q}}V_{q},$$
where each $V_{j}$ is a MDC-Schottky extension group of rank $1$ and $W_{j}=V_{j-1} \cap V_{j}$ is either trivial or a cyclic group of order $d \in \{2,3,4,6\}$.

All of the above permits us to see that 
$$K=G*K^{0}*_{\langle t_{2} \rangle} K^{2} *_{\langle t_{3} \rangle} K^{3} *_{\langle t_{4} \rangle} \cdots *_{\langle t_{s} \rangle} K^{s},$$
where $G$ is either trivial or an MDC-Schottky group of some rank and the groups $K^{j}$ ($j\geq 2$) are MDC-Schottky extension groups of rank $1$.

\subsubsection{Second case}
Let us now assume that in the complement of the collection ${\mathcal F}$ there is a $H$-invariant component $\widehat{R}_{0}$ of genus $1$, that is, $H$ is a finite group of automorphisms of a genus one surface. We work similarly as before, but now $K^{0}$ may be any of the MDC-Schottky extension groups of rank $1$.

\subsubsection{Third case}
Let us assume now that in the complement of the collection ${\mathcal F}$ there is no $H$-invariant component $\widehat{R}_{0}$. In this case, there is an invariant loop $\delta \in{\mathcal F}$ and the group $H$ is either the trivial group or a cyclic group of order $2$, $3$, $4$ or $6$, or a dihedral group of order either $4$, $6$, $8$ or $12$ (with one of the involutions permuting both components in $S \setminus \delta$). In this case, we set as $\widehat{R}_{0}$ the genus $2$ surface obtained by the union of these two tori components with $\delta$ as a common boundary loop, which is $H$-invariant. We may proceed as in the first case, but $K^{0}$  is now any of the MDC-Schottky extension group of rank $2$.

\section{Proof of Theorem \ref{cor1}}\label{Sec:prueba2}
Let $M$ be an MDC-handlebody of genus $g \geq 2$ with an MDC-Schottky structure produced by the MDC-Schottky group $G$. Let $H<{\rm Aut}^{+}(M)$ and $K$ be the MDC-Schottky extension group obtained by lifting $H$ to the universal cover space. Let $S$ be the conformal boundary of $M$ with such a structure. Theorem \ref{loops} provides a collection ${\mathcal F}$ of loops which is $H$-invariant and satisfies all the properties in there stated.
Each of the loops in ${\mathcal F}$ is the boundary of a compressing disc in $M$, all of them pairwise disjoint and $H$ invariant (this is exactly what the Equivariant Loop Theorem \cite{M-Y} asserts).

\noindent
(1) For $g \geq 3$, the collection ${\mathcal F}$ has the property that either: 
\begin{itemize}[leftmargin=15pt]
\item[(a)] there is a $H$-invariant component $\widehat{R}_{0} \subset S \setminus {\mathcal F}$ of genus $\gamma \in \{0,1\}$ and all other components are of genus $1$; or 
\item[(b)] $S \setminus {\mathcal F}$ has no $H$-invariant component, each component is of genus $1$ and there is a loop $\delta \in {\mathcal F}$ which is $H$-invariant.
\end{itemize}

 In case (a) the boundary loops of $\widehat{R}_{0}$ is sub-collection of ${\mathcal F}$ which still $H$-invariant.  If $\gamma=0$, then the group $H$ is a finite group isomorphic to either the trivial one or a cyclic group or a dihedral group or ${\mathcal A}_{4}$ or ${\mathcal A}_{5}$ or ${\mathfrak S}_{4}$. If $\gamma=1$, then $H$ is a finite group of conformal automorphisms of a genus one Riemann surface.
 
In case (b) the group $H$ is either trivial or a cyclic group of order $2$, $3$, $4$ or $6$ keeping invariant each of the two components of $S \setminus \delta$ or it is a cyclic group of order two permuting both of these two components or a dihedral group of order either $4$, $6$, $8$ or $12$ (containing one involution that permutes both of the two components and the maximal cyclic group preserving each of them).

\noindent
(2) For $g=2$, the collection ${\mathcal F}$ consists of exactly one loop $\delta$ that divides $S$ into two genera one tori. Moreover, $H$ keeps invariant it. It follows that $H$ is either trivial, cyclic group of order $2$, $3$, $4$ or $6$ that keep invariant each of the two tori or cyclic of oder two that interchanges both tori or a dihedral group of order either $4$, $6$, $8$ or $12$ (containing one involution that permutes both tori and the maximal cyclic group preserving each of them).

\section{Proof of Corollary \ref{maxteo}}\label{Sec:proofmaxteo}
Let us assume $K$ is an MDC-Schottky extension group of rank $g \geq 2$ and $G \lhd K$ is an MDC-Schottky group of rank $g$ and finite index. We already noted that the index is bounded above by $12(g-1)$. Let us assume the index is exactly $12(g-1)$. Let us denote by $S$ the closed Riemann surface uniformized by $G$ and let $H=K/G$ be the induced group of automorphisms of $S$ of order $12(g-1)$. 
The case $g=2$ is a direct consequence of Theorem \ref{cor1}; the only case for which $|H|=12$ is the one provided in Remark \ref{ejemplo1}.
Let us assume $g \geq 3$ and consider a $H$-invariant collection of loops ${\mathcal F}\subset S$ as in Theorem \ref{loops}. 

(1) Assume $S \setminus {\mathcal F}$ has a $H$-invariant surface
of genus $1$. This will provide a genus $1$ Riemann surface admitting a group of conformal automorphisms of order $12(g-1)$ keeping invariant a collection of at most $g-1$ points; which is not possible for $g \geq 2$. If there is a loop in ${\mathcal F}$ which is $H$-invariant, then the highest order for $H$ is $12$; this cannot be $12(g-1)$ if $g \geq 3$.

(2) Let us now assume there is a planar surface $R_{0}\subset S \setminus {\mathcal F}$ which is $H$-invariant. In this case,  
$H$ is isomorphic to either (i) ${\mathbb Z}_{12(g-1)}$ or (ii) $D_{6(g-1)}$ or (iii) ${\mathcal A}_{4}$ or (iv) ${\mathcal A}_{5}$ or (v) ${\mathfrak S}_{4}$. Also, 
there is a sub-collection of loops, say $\alpha_{1},\ldots,\alpha_{r} \in {\mathcal F}$, where $r \leq g$, which is also $H$-invariant and all of them are exactly the borders of $R_{0}$. Moreover, each $\alpha_{j}$ (for $j=1,\ldots,r$) has either trivial or cyclic $H$-stabilizer whose order is either $2,3,4$ or $6$.
If $g \geq 7$, then $H$ may only be either ${\mathbb Z}_{12(g-1)}$ or a $D_{6(g-1)}$. In the cyclic case, $g$ should satisfy the inequality $g \geq r=12(g-1)L$, some $L \geq 0$, a contradiction. In the dihedral case, $g$ should satisfy an inequality either of the previous form (then a contradiction) or one of the form $g \geq r=12(g-1)L_{1} + 6(g-1)L_{2}$ for some $L_{1} \geq 0$ and $L_{2} \in \{0,1,2\}$; again a contradiction.
If $g=6$, then $|H|=60$ and $H$ can be either ${\mathcal A}_{5}$, or $D_{30}$ or ${\mathbb Z}_{60}$. In the cyclic case, $g$ must satisfy the inequality $g \geq r=60L$, for some $L \geq 0$, a contradiction. In the dihedral case, $g$ must satisfy an equality
of the form $g \geq r=60L_{1} + 30L_{2}$ for some $L_{1} \geq 0$ and $L_{2} \in \{0,1,2\}$; again a contradiction. In the ${\mathcal A}_{5}$ there is no possible way to draw a system of $r \leq g$ pairwise disjoint simple loops as required and being invariant under such a group.
If $g=5$, the $|H|=48$ and $H$ can be either  $D_{24}$ or ${\mathbb Z}_{48}$. In the cyclic case, $g$ must satisfy the inequality $g \geq r=48L$, for some $L \geq 0$, a contradiction. In the dihedral case, $g$ must satisfy an inequality
of the form $g \geq r=48L_{1} + 24L_{2}$ for some $L_{1} \geq 0$ and $L_{2} \in \{0,1,2\}$; again a contradiction. 
If $g=4$, the $|H|=36$ and $H$ can be either  $D_{18}$ or ${\mathbb Z}_{36}$. In the cyclic case, $g$ must satisfy the inequality $g \geq r=36L$, for some $L \geq 0$, a contradiction. In the dihedral case, $g$ must satisfy an inequality
of the form $g \geq r=36L_{1} + 18L_{2}$ for some $L_{1} \geq 0$ and $L_{2} \in \{0,1,2\}$; again a contradiction. 
If $g=3$, then $|H|=24$ and $H$ can be either ${\mathfrak S}_{4}$, or $D_{12}$ or ${\mathbb Z}_{24}$. In the cyclic case, $g$ must satisfy the inequality $g \geq r=24L$, for some $L \geq 0$, a contradiction. In the dihedral case, $g$ must satisfy an inequality
of the form $g \geq r=24L_{1} + 12L_{2}$ for some $L_{1} \geq 0$ and $L_{2} \in \{0,1,2\}$; again a contradiction. In the ${\mathfrak S}_{4}$ there is no possible way to draw a system of $r \leq g$ pairwise disjoint simple loops as required and being invariant under such a group.

\section{An example}
Let $K$ be an MDC-Schottky extension group $K$ contains the MDC-Schottky group $G$ of rank $g \geq 3$ as a finite index normal subgroup so that $H=K/G$ acts on $S=\Omega(G)/G$. Let us assume the collection of loops ${\mathcal F}$ (provided by Theorem \ref{loops}) defines a genus zero surface $\widehat{R}_{0}$ which is invariant under $H$. Also, we assume the collection of surfaces $\widehat{R}_{j}$ (as in the proof above) are of genus 
$g_{j}=1$ for all $j=1,\ldots,r$; in which case $r=g$. In this case, each group $K^{m}$  is conjugated  (in the M\"obius group) to either one of the $K_{1}$, $K_{2}$, $K_{22}$, $K_{3}$, $K_{4}$ or $K_{6}$. 
In the case the $K^{m}$ is conjugate to $K_{d}$, for $d \in \{2,3,4,6\}$, we set $d_{m}=d$ (if $K_{22}$, then set $d_{m}=2$). It follows, from the proof (first case) that
$$K \cong \underbrace{{\mathbb Z}^{2}*\cdots *{\mathbb Z}^{2}}_{u} *K_{0}*_{{\mathbb Z}_{d_{1}}} K_{d_{1}}*_{{\mathbb Z}_{d_{2}}} K_{d_{2}}*_{{\mathbb Z}_{d_{3}}} \cdots *_{{\mathbb Z}_{d_{s-u}}} K_{d_{s-u}}, \; {\rm where} \; K^{0} \cong H,$$
and $$g=|K_{0}| \left(u+\sum_{j=1}^{s-u} \frac{1}{d_{j}} \right).$$

Next, we proceed to see the signatures of the orbifolds uniformized by $K$.

\subsection{}
If $K_{0}$ is trivial, then $K$ is an MDC-Schottky group of rank $g$ and it uniformizes a genus $g$ Riemann surface.

\subsection{}
Assume $K_{0} \cong {\mathbb Z}_{n}$. 
If $n \in \{2,3,4,6\}$, then we may first amalgamate $K_{0}$ with $a \in \{0,1,2\}$ groups isomorphic to $K_{n}$ (otherwise $a=0$). Next, we may do the free product of it with an MDC-Schottky group of rank $r$. In this case, $g=nr+a$ and $K$ uniformizes an orbifold with signature
$$
\begin{array}{|c|c||c|c|}\hline 
(r,2;n,n) & \mbox{if $a=0$} &
(r,4;2,2,2,2) & \mbox{$n=2$ and $a=1$} \\\hline 
(r,6;2,2,2,2,2,2) & \mbox{$n=2$ and $a=2$} &
(r,3;3,3,3) & \mbox{$n=3$ and $a=1$} \\\hline 
(r,4;3,3,3,3) & \mbox{$n=3$ and $a=2$} & 
(r,4;2,2,4,4) & \mbox{$n=4$ and $a=1$} \\\hline 
(r,6;2,2,2,2,4,4) & \mbox{$n=4$ and $a=2$} & 
(r,3;2,3,6) & \mbox{$n=6$ and $a=1$} \\\hline 
(r,4;2,2,3,3) & \mbox{$n=6$ and $a=2$} \\\hline 
\end{array}
$$


\subsection{}
Assume $K_{0}\cong D_{n}$.
We may first amalgamate $K_{0}$ with $b \in \{0,1,2\}$ groups isomorphic to $K_{2}$ (if $n$ odd, necessarily $b=2$). For $n \in \{2,3,4,6\}$ we may also amalgamate to it $a \in \{0,1\}$ groups isomorphic to $K_{n}$. Then, we may do the free product of it with an MDC-Schottky group of rank $r$. In this case, $g=2nr+2a+bn$ and $K$ uniformizes an orbifold with signature
$$
\begin{array}{|c|c|}\hline 
(r,3;2,2,n) & \mbox{$a+b=0$} \\\hline
(r,5;2,2,2,2,n) & \mbox{$a+b=1$, $n$ even and $a=0$ if $n \notin \{2,4\}$} \\\hline
(r,7;2,2,2,2,2,2,n) & \mbox{$a+b=2$ and $a=0$ if $n \notin \{2,4\}$} \\\hline
(r,9;2,2,2,2,2,2,2,2,2) & \mbox{$n=2, a+b=3$} \\\hline
(r,4;2,2,3,3) & \mbox{$n=3, a=1, b=0$} \\\hline
(r,8;2,2,2,2,2,2,3,3) & \mbox{$n=3, a=1, b=2$} \\\hline 
(r,9;2,2,2,2,2,2,2,2,4) & \mbox{$n=4, a=1, b=2$} \\\hline
(r,4;2,2,2,3) & \mbox{$n=6, a=1, b=0$} \\\hline
(r,6;2,2,2,2,2,3) & \mbox{$n=6, a=b=1$} \\\hline
(r,8;2,2,2,2,2,2,2,3) & \mbox{$n=6, a=1, b=2$} \\\hline
\end{array}
$$


\subsection{}
Assume $K_{0} \cong {\mathcal A}_{4}$.
We may first amalgamate $K_{0}$ with $a\in \{0,1,2\}$ groups isomorphic to $K_{3}$ and $b\in \{0,1\}$ groups isomorphic to $K_{2}$. Then, we may do the free product of it with an MDC-Schottky group of rank $r$. Then,  $g=12r+4a+6b$, where $a\in \{0,1,2\}$ and $b \in \{0,1\}$ and $K$ uniformizes an orbifold with signature
$$
\begin{array}{|c|c||c|c|}\hline
(r,3;2,3,3) & \mbox{$a=b=0$} &
(r,4;2,3,3,3) & \mbox{$a=1$, $b=0$} \\\hline
(r,5;2,3,3,3,3) & \mbox{$a=2$,  $b=0$} &
(r,5;2,2,2,3,3) & \mbox{$a=0, b=1$} \\\hline
(r,6;2,2,2,3,3,3) & \mbox{$a=1, b=1$} &
(r,7;2,2,2,3,3,3,3) & \mbox{$a=2, b=1$} \\\hline
\end{array}
$$


\subsection{}
Assume $K_{0} \cong {\mathcal A}_{5}$.
We may amalgamate $K_{0}$ with $a \in \{0,1\}$ groups isomorphic to $K_{3}$ and $b \in \{0,1\}$ groups isomorphic to $K_{2}$. Then, we may do the free product of it with an MDC-Schottky group of rank $r$. Then, $g=60r+20a+30b$, where $a,b\in \{0,1\}$ and $K$ uniformizes an orbifold with signature
$$
\begin{array}{|c|c||c|c|}\hline
(r,3;2,3,5) & \mbox{$a=b=0$} & 
(r,5;2,2,2,3,5) & \mbox{$a=0$, $b=1$} \\\hline
(r,4;2,3,3,5) & \mbox{$a=1$,  $b=0$} &
(r,6;2,2,2,3,3,5) & \mbox{$a=1, b=1$} \\\hline
\end{array}
$$


\subsection{}
Assume $K_{0}\cong {\mathfrak S}_{4}$.
We may amalgamate $K_{0}$ with 
$a\in \{0,1\}$  groups isomorphic to $K_{4}$ and $b \in \{0,1\}$
groups isomorphic to $K_{2}$ and $c \in \{0,1\}$ groups isomorphic to $K_{3}$r. Then, we may do the free product of it with an MDC-Schottky group of rank $r$. In this case $g=24r+6a+12b+8c$, where $a,b,c\in \{0,1\}$, and $K$ uniformizes an orbifold with signature
$$
\begin{array}{|c|c||c|c|}\hline
(r,3;2,2,4) & \mbox{$a=b=c=0$} & 
(r,4;2,3,3,4) & \mbox{$a=b=0$, $c=1$} \\\hline
(r,5;2,2,2,3.4) & \mbox{$a=c=0$, $b=1$} &
(r,6;2,2,2,3,3,4) & \mbox{$a=0, b=c=1$} \\\hline
(r,5;2,2,2,3,4) & \mbox{$a=1, b=c=0$} &
(r,7;2,2,2,2,2,3,4) & \mbox{$a=b=1, c=0$} \\\hline
(r,6;2,2,2,3,3,4) & \mbox{$a=c=1, b=0$} &
(r,8;2,2,2,2,2,3,3,4) & \mbox{$a=b=c1$} \\\hline
\end{array}
$$



\end{document}